\documentclass[envcountsame]{llncs}
\usepackage{proof,amsmath,amssymb,shortdedic}

\title{Effective Finite-Valued Approximations of General Propositional Logics}

\author{Matthias Baaz\inst{1} 
\and Richard Zach\inst{2}\thanks{Research supported by the Natural Sciences and Engineering Research Council of Canada.}}

\institute{Technische Universit\"at Wien,  Institut f\"ur Diskrete
Mathematik und\\ Geometrie E104, A--1040 Vienna,
Austria. \email{baaz@logic.at} \and University of Calgary, Department
of Philosophy,\\ Calgary, Alberta T2N 1N4,
Canada. \email{rzach@ucalgary.ca}}

\spnewtheorem{defn}[theorem]{Definition}{\bfseries}{\upshape}

\let\impl\supset
\let\Box\square

\def\I{\ensuremath{v}}
\def\J{\ensuremath{w}}

\def\tbox{\widetilde{\vphantom{l}\Box}}
\def\tdiamond{\widetilde{\vphantom{l}\Diamond}}
\def\MC{{\rm MC}}

\def\depth{\mathrm{dp}}
\def\CL{{\bf CL}}
\def\LL{{\bf LL}}
\def\IPC{{\bf IPC}}
\def\A{{\bf A}}
\def\IPL{{\bf IPL}}
\def\Taut{{\rm Taut}}
\def\Thm{{\rm Thm}}
\def\Frm{{\rm Frm}}
\def\Var{\mathord{\rm Var}}
\def\M{{\bf M}}
\def\AL{{\bf L}}
\def\AK{{\bf C}}
\def\SF{{\bf S4}}
\def\G{{\bf G}}
\def\Gm{{\ensuremath{{\bf G}_m}}}
\def\LA{\ensuremath{{\cal L}}}
\def\K{\ensuremath{{\cal K}}}
\def\MVL{{\rm MVL}}
\def\better{\mathrel{\lhd}}
\def\bettereq{\mathrel{\unlhd}}

\def\cbi{\mathop{\tilde{\leftrightarrow}}}
\def\Nex{{\scriptstyle \bigcirc}}

\begin{document}
\bibliographystyle{splncs}
\maketitle
\pagestyle{plain}

\begin{shortdedication}
Dedicated to Professor Trakhtenbrot on the occasion of his 85th birthday.
\end{shortdedication}

\begin{abstract}
Propositional logics in general, considered as a set of sentences, can
be undecidable even if they have ``nice'' representations, e.g., are
given by a calculus.  Even decidable propositional logics can be
computationally complex (e.g., already intuitionistic logic is
PSPACE-complete).  On the other hand, finite-valued logics are
computationally relatively simple---at worst NP.  Moreover,
finite-valued semantics are simple, and general methods for theorem
proving exist.  This raises the question to what extent and under what
circumstances propositional logics represented in various ways can be
approximated by finite-valued logics.  It is shown that the minimal
$m$-valued logic for which a given calculus is strongly sound can be
calculated.  It is also investigated under which conditions
propositional logics can be characterized as the intersection of
(effectively given) sequences of finite-valued logics.
\end{abstract}

\section{Introduction}

The question of what to do when faced with a new logical
calculus is an old problem of mathematical logic.  Often, at
least at first, no semantics are available.  For example, intuitionistic
propositional logic was constructed by Heyting only as a calculus;
semantics for it were proposed much later.  Linear logic was in a
similar situation in the early 1990s.  The lack of semantical methods
makes it difficult to answer questions such as: Are statements of a
certain form (un)derivable? Are the axioms independent? Is the
calculus consistent?  For logics closed under substitution, many-valued
methods have often proved valuable since they were first used for
proving underivabilities by Bernays \cite{Bernays:26} in 1926 (and
later by others, e.g., McKinsey and Wajsberg; see also
\cite[\S~25]{Rescher:69}).  The method is very simple.  Suppose
you find a many-valued logic in which the axioms of a given calculus
are tautologies, the rules are sound, but the formula in question is not
a tautology: then the formula cannot be derivable.

\begin{example}
Intuitionistic propositional logic is axiomatized by the following
calculus~\IPC:
\begin{enumerate}
\item Axioms:
\[
\begin{array}{ll}
a_1 & A \impl A \land A \\
a_2 & A \land B \impl B \land A \\
a_3 & (A \impl B) \impl (A \land C \impl B \land C) \\
a_4 & (A \impl B) \land (B \impl C) \impl (A \impl C) \\
a_5 & B \impl (A \impl B) \\
a_6 & A \land (A \impl B) \impl B \\
a_7 & A \impl A \lor B \\
a_8 & A \lor B \impl B \lor A \\
a_9 & (A \impl C) \land (B \impl C) \impl (A \lor B \impl C) \\
a_{10} & \neg A \impl (A \impl B) \\
a_{11} & (A \impl B) \land (A \impl \neg B) \impl \neg A \\
a_{12} & A \impl (B \impl A \land B)
\end{array}
\]
\item Rules (in usual notation):
\[
\infer[$MP$]{B}{A & A \impl B}
\]
\end{enumerate}
Now consider the two-valued logic with classical truth tables, except that
$\lnot$ maps both truth values to ``true''.  Then every axiom except $a_{10}$
is a tautology and modus ponens preserves truth. Hence $a_{10}$ is independent 
of the other axioms.
\end{example}

To use this method to answer underivability question in general it is
necessary to find many-valued matrices for which the given calculus is
sound.  It is also necessary, of course, that the matrix has as few
tautologies as possible in order to be useful.  We are interested
in how far this method can be automatized.

Such ``optimal'' approximations of a given calculus may also have
applications in computer science. In the field of artificial
intelligence many new (propositional) logics have been introduced.
They are usually better suited to model the problems dealt with in AI
than traditional (classical, intuitionistic, or modal) logics, but
many have two significant drawbacks: First, they are either given
solely semantically or solely by a calculus.  For practical purposes,
a proof theory is necessary; otherwise computer representation of and
automated search for proofs/truths in these logics is not feasible.
Although satisfiability in many-valued propositional logics is (as in
classical logic) NP-complete \cite{Mundici:87}, this is still
(probably) much better than many other important logics.

On the other hand, it is evident from the work of
Carnielli~\cite{Carnielli:87} and H\"ahnle~\cite{Hahnle:93} on
tableaux, and Rousseau, Takahashi, and Baaz et
al.~\cite{BFSZ:98} on sequents, that finite-valued logics are,
from the perspective of proof {\em and} model theory, very close to
classical logic.  Therefore, many-valued logic is a very suitable
candidate if one looks for approximations, in some sense, of given
complex logics.

What is needed are methods for obtaining finite-valued approximations
of the propositional logics at hand.  It turns out, however, that a
shift of emphasis is in order here.  While it is the {\em logic} we
are actually interested in, we always are given only a {\em
representation} of the logic.  Hence, we have to concentrate on
approximations of the representation, and not of the logic per se.

What is a representation of a logic?  The first type of representation
that comes to mind is a calculus.  Hilbert-type calculi are the
simplest conceptually and the oldest historically.  We will
investigate the relationship between such calculi on the one hand and
many-valued logics or effectively enumerated sequences of many-valued
logics on the other hand.  The latter notion has received considerable
attention in the literature in the form of the following two problems:
Given a calculus~\AK,
\begin{enumerate}
\item find a minimal (finite) matrix for which \AK{} is sound
(relevant for non-derivabil\-ity and independence proofs), and
\item find a sequence of finite-valued logics, preferably effectively enumerable,
whose intersection equals the theorems of~\AK, and its converse,
given a sequence of finite-valued logics, find a calculus
for its intersection (exemplified by Ja{\'s}kowski's
sequence for intuitionistic propositional calculus, and by
Dummett's extension axiomatizing the intersection of the sequence
of G\"odel logics, respectively).
\end{enumerate}
For (1), of course, the best case would be a finite-valued logic~\M{}
whose tautologies {\em coincide} with the theorems of~\AK.  \AK{} then
provides an axiomatization of~\M.  This of course is not always
possible, at least for {\em finite}-valued logics.  Lindenbaum
\cite[Satz~3]{LukasiewiczTarski:30} has shown that any logic (in our
sense, a set of formulas closed under substitution) can be
characterized by an {\em infinite}-valued logic.  For a discussion of
related questions see also Rescher~\cite[\S~24]{Rescher:69}.

In the following we study these questions in a general setting.
Consider a propositional Hilbert-type calculus~\AK.  It is (weakly)
sound for a given $m$-valued logic if all its theorems are
tautologies.  Unfortunately, it turns out that it is undecidable if a
calculus is sound for a given $m$-valued logic.  However, for natural
stronger soundness conditions this question is decidable; a
finite-valued logic for which \AK{} satisfies such soundness
conditions is called a \emph{cover} for~\AK.  The optimal (i.e., minimal
under set inclusion of the tautologies) $m$-valued cover for~\AK{} can
be computed.  The next question is, can we find an approximating
sequence of $m$-valued logics in the sense of~(2)?  It is shown that
this is impossible for undecidable calculi~\AK, and possible for all
decidable logics closed under substitution.  This leads us to the
investigation of the {\em many-valued closure} $\MC(\AK)$ of~\AK,
i.e., the set of formulas which are true in all covers of~\AK. In
other words, if some formula can be shown to be underivable in~\AK{}
by a Bernays-style many-valued argument, it is not in the many-valued
closure.  Using this concept we can classify calculi according to
their many-valued behaviour, or according to how good they can be dealt
with by many-valued methods.  In the best case $\MC(\AK)$ equals the
theorems of~\AK{} (This can be the case only if \AK{} is decidable).
We give a sufficient condition for this being the case.
Otherwise $\MC(\AK)$ is a proper superset of the theorems of~\AK.

Axiomatizations \AK~and~$\AK'$ of the same logic may have different
many-valued closures $\MC(\AK)$ and $\MC(\AK')$ while being
model-theoretically indistinguishable. Hence, the many-valued closure
can be used to distinguish between \AK~and~$\AK'$ with regard to their
proof-theoretic properties.

Finally, we investigate some of these questions for other
representations of logics, namely for decision procedures and
(effectively enumerated) finite Kripke models.  In these cases
approximating sequences of many-valued logics whose intersection
equals the given logics can always be given.

Some of our results were previously reported in \cite{BaazZach:94}, of
which this paper is a substantially revised and expanded version.

\section{Propositional Logics}

\begin{defn}
A {\em propositional language~$\cal L$} consists of the following:
\begin{enumerate}
\item propositional variables:  $X_1$, $X_2$, $X_3$, \dots
\item propositional connectives of arity $n_j$:
$\Box_1^{n_1}$,~$\Box_2^{n_2}$, \dots,~$\Box_r^{n_r}$. If $n_j = 0$,
then $\Box_j$ is called a {\em propositional constant}.
\item Auxiliary symbols: $($, $)$, and $,$ (comma).
\end{enumerate}
\end{defn}

Formulas and subformulas are defined as usual.  We denote the set of
formulas over a language $\cal L$ by $\Frm(\LA)$.  By $\Var(A)$
we mean the set of propositional variables occurring in~$A$.  A
\emph{substitution}~$\sigma$ is a mapping of variables to formulas,
and if $F$ is a formula, $F\sigma$ is the result of simultaneously
replacing each variable $X$ in $F$ by $\sigma(X)$.

\begin{defn}
The \emph{depth} $\depth(A)$ of a formula~$A$ is defined as follows:
$\depth(A) = 0$ if $A$ is a variable or a 0-place connective
(constant).  If
$A = \Box(A_1, \ldots, A_n)$, then let
$\depth(A) = \max\{\depth(A_1), \ldots, \depth(A_n)\} + 1$.
 \end{defn}

\begin{defn}\label{defn:proplogic}
A {\em propositional Hilbert-type calculus}~\AK{} in the language $\LA$ is given by
\begin{enumerate}
\item a finite set $A(\AK) \subseteq \Frm(\LA)$ of axioms.
\item a finite set $R(\AK)$ of rules of the form
\[
\infer[r]{C}{A_1 & \ldots & A_n}
\]
where $C$, $A_1$, \dots, $A_n \in \Frm(\LA)$
\end{enumerate}
A formula~$F$ is a {\em theorem} of \AL{} if there is a derivation of $F$
in \AK, i.e., a finite sequence \[ F_1, F_2, \ldots, F_s = F \]
of formulas s.t.{} for each $F_i$ there is a substitution $\sigma$ so that either
\begin{enumerate}
\item $F_i = A\sigma$ where $A$ is an axiom in $A(\AK)$, or
\item there are $F_{k_1}$, \dots,~$F_{k_n}$ with $k_j < i$
and a rule $r \in R(\AK)$ with premises $A_1$, \dots, $A_n$ and conclusion $C$, s.t.{} $F_{k_j} = A_j\sigma$
and $F_i = C\sigma$
\end{enumerate}
If $F$ is a theorem of~\AK{} we write $\AK \vdash F$.
The set of theorems of~\AK{} is denoted by~$\Thm(\AK)$.
\end{defn}

\begin{remark}
The above notion of a propositional rule is the one usually
used in axiomatizations of propositional logic.  It is, however,
by no means the only possible notion.  For instance, Sch\"utte's
rules
\[
\infer{A(X)}{A(\top) & A(\bot)} \qquad
\infer{A(C) \leftrightarrow A(D)}{C \leftrightarrow D}
\]
where $X$ is a propositional variable, and $A$, $C$, and $D$ are
formulas, does not fit under the above definition.  And not only do
they not fit this definition, the proof-theoretic behaviour of such
rules is indeed significantly different from other ``ordinary'' rules.
For instance, the rule on the left allows the derivation of all
tautologies with $n$ variables in number of steps linear in~$n$; with
a Hilbert-type calculus falling under the definition, this is not
possible~\cite{BaazZach:94a}.
\end{remark}

\begin{remark}\label{rem:sequent}
Many logics are more naturally axiomatized using sequent calculi, in
which structure (sequences of formulas, sequent arrows) are used in
addition to formulas.  Many sequent calculi can easily be encoded in
Hilbert-type calculi in an extended language, or even
straightforwardly translated into Hilbert calculi in the same
language, using constructions sketched below:
\begin{enumerate}
\item Sequences of formulas can be coded using a binary
operator~$\cdot$.  The sequent arrow can simply be coded as a binary
operator~$\to$.  For empty sequences, a constant $\Lambda$ is used.
We have the following rules, to assure associativity
of~$\cdot$:
\[
\infer{X \cdot \Bigl(\bigl((U\cdot V)\cdot W\bigr)\cdot Y\Bigr) \to Z}
      {X \cdot \Bigl(\bigl(U\cdot (V\cdot W)\bigr)\cdot Y\Bigr) \to Z}
\qquad
\infer{\Bigl(X \cdot \bigl((U\cdot V)\cdot W\bigr)\Bigr)\cdot Y \to Z}
      {\Bigl(X \cdot \bigl(U\cdot (V\cdot W)\bigr)\Bigr)\cdot Y \to Z}
\]
as well as the respective rules without $X$, without $Y$, without
both $X$~and~$Y$, with the rules upside-down, and also for the
right side of the sequent (20~rules total).
\item The usual sequent rules can be coded using the above
constructions, e.g., the $\land$-Right rule of \textbf{LJ} would become:
\[
\infer{U \to V \cdot (X \land Y)}{U \to V \cdot X & U \to V \cdot Y} 
\]
\item If the language of the logic in question contains constants and
connectives which ``behave like'' the $\Lambda$ and $\cdot$ on the
left or right of a sequent, and a conditional which behaves like the
sequent arrow, then no additional connectives are necessary. For
instance, instead of $\cdot$, $\Lambda$ on the left, use $\land$,
$\top$; on the right, use $\lor$, $\bot$, and use $\impl$ instead of
$\to$. Addition of the rule
\[
\infer{X}{\top \impl X}
\]
would then result in a calculus which proves exactly the formulas~$F$
for which the sequent ${} \to F$ is provable in the original sequent
calculus.
\item Some sequent rules require restrictions on the
form of the side formulas in a rule, e.g., the
$\Box$-right rule in modal logics:
\[
\infer{\Box\Pi \to \Box A}{\Box\Pi \to A}
\]
It is not immediately possible to accommodate such a rule in the
translation. However, in some cases it can be replaced with another
rule which can.  E.g., in \SF, it can be replaced by
\[
\infer{\Box\Pi \to \Box A}{\Pi \to A}
\]
which can in turn be accommodated using rules such as
\[
\infer{\Box X \impl \Box Y}{X \impl Y} \quad 
\infer{U \land (\Box X \land \Box Y) \impl V}{U \land \Box(X \land Y) \impl V}
\quad 
\infer{U \land \Box Y \impl V}{U \land \Box\Box Y \impl V}
\]
(in the version with standard connectives serving as $\cdot$ and sequent arrow).
\end{enumerate}
\end{remark}

\begin{defn}\label{defn:analytic}
A propositional Hilbert-type calculus is called {\em strictly
analytic} iff for every rule
\[
\infer[r]{C}{A_1 \ldots A_n}
\]
it holds that $\Var(A_i) \subseteq \Var(C)$ and $\depth(A_i\sigma) \le \depth(C\sigma)$ 
for every substitution~$\sigma$.
\end{defn}

This notion of strict analyticity is orthogonal to the one employed in
the context of sequent calculi, where ``analytic'' is usually taken to
mean that the rules have the subformula property (the formulas in the
premises are subformulas of those in the conclusion).  A strictly
analytic calculus in our sense need not satisfy this.  On the other
hand, Hilbert calculi resulting from sequent calculi using the coding
above need not be strictly analytic in our sense, even if the sequent
calculus has the subformula property.  For instance, the contraction
rule does not satisfy the condition on the depth of substitution
instances of the premises and conclusion.  The standard notion of
analyticity does not entail decidability, since for instance cut-free
propositional linear logic~\LL{} is analytic but \LL{} is
undecidable~\cite{LMSS:90}.  Our notion of strict analyticity does
entail decidability, since the depth of the conclusion of a rule in a
proof is always greater or equal to the depth of the premises, and so
the number of formulas that can appear in a proof of a given formula
is finite.

\begin{defn}
A {\em propositional logic}~\AL{} in the language~$\cal L$ is a
subset of~$\Frm(\LA)$ closed under substitution.
\end{defn}

Every propositional calculus~\AK{} defines a propositional logic,
namely $\Thm(\AK)$, since
$\Thm(\AK)$ is closed under substitution.
Not every propositional logic, however, is axiomatizable, let
alone finitely axiomatizable by a Hilbert calculus.
For instance, the logic
\begin{eqnarray*}
\{\Box^k(\top) & \mid & \hbox{$k$ is the G\"odel number of a} \\
& & \hbox{true sentence of arithmetic} \}
\end{eqnarray*}
is not axiomatizable, whereas the logic
\[
\{\Box^k(\top) \mid \hbox{$k$ is prime}\}
\]
is certainly axiomatizable (it is even decidable), but not by
a Hilbert calculus using only $\Box$ and $\top$. (It is easily
seen that any Hilbert calculus for $\Box$ and $\top$ has
either only a finite number of theorems or yields arithmetic
progressions of~$\Box$'s.)

\begin{defn}
A {\em propositional finite-val\-ued logic \M} is given by a finite
set of truth values $V(\M)$, the set of {\em designated truth values}
$V^+(\M) \subseteq V(\M)$, and a set of truth functions $\tbox_j
\colon V(\M)^{n_j} \to V(\M)$ for all connectives $\Box_j \in \LA$
with arity~$n_j$.
\end{defn}

\begin{defn}
A {\em valuation} \I{} is a mapping from the set of propositional
variables into~$V(\M)$.  A valuation \I{} can be extended in the standard
way to a function from formulas to truth values.  \I{} {\em
satisfies} a formula $F$, in symbols:  $\I \models_\M F$, if $\I(F)
\in V^+(\M)$.  In that case, \I{} is called a {\em model} of~$F$,
otherwise a {\em countermodel}.  A formula $F$ is a {\em tautology} of
\M{} iff it is satisfied by every valuation. Then we write $\M \models F$.
We denote the set of tautologies of \M{} by $\Taut(\M)$.
\end{defn}

\begin{example}
The sequence of $m$-valued G\"odel logics~\Gm{} is given by $V(\Gm) =
\{0, 1, \ldots, m-1\}$, the designated values $V^+(\Gm) = \{0\}$, and
the following truth functions:
\begin{eqnarray*}
\widetilde{\neg}_\Gm(v) & = & \begin{cases} 0 & \text{for $v = m-1$} \\ m-1 & \text{for $v \neq m-1$} \end{cases}\\
\widetilde{\lor}_\Gm(v, w) & = & \min(a, b) \\
\widetilde{\land}_\Gm(v, w) & = & \max(a, b) \\
\widetilde{\impl}_\Gm(v, w) & =& \begin{cases}0 & \text{for $v \ge w$} \\ w & \text{for $v < w$} \end{cases}
\end{eqnarray*}
\end{example}

In the remaining sections, we will concentrate on the relations
between propositional logics~\AL{} represented in some way (e.g., by a
calculus), and finite-valued logics~\M.  The objective is to find
many-valued logics~\M{}, or effectively enumerated sequences thereof,
which, in a sense, approximate the the logic~\AL.

The following well-known product construction is useful for characterizing
the ``intersection'' of many-valued logics.

\begin{defn}
Let $\M$ and $\M'$ be $m$ and $m'$-valued logics, respectively.  Then
$\M \times \M'$ is the $mm'$-valued logic where $V(\M \times \M') =
V(\M) \times V(\M')$, $V^+(\M \times \M') = V^+(\M) \times V^+(\M')$,
and truth functions are defined component-wise.  I.e., if $\Box$ is an
$n$-ary connective, then \[\tbox_{\M \times \M'}(w_1, \ldots, w_n) = 
\langle \tbox_\M(w_1, \ldots, w_n),\tbox_{\M'}(w_1, \ldots, w_n)\rangle.\]
\end{defn}

For convenience, we define the following:  Let \I{} and $\I'$ be
valuations of \M{} and $\M'$, respectively.  $\I \times \I'$ is the
valuation of $\M \times \M'$ defined by:  $(\I \times \I')(X) =
\langle \I(X), \I'(X) \rangle$.  If $\I^\times$ is a valuation of $\M
\times \M'$, then the valuations $\pi_1\I^\times$ and $\pi_2\I^\times$
of $\M$ and $\M'$, respectively, are defined by $\pi_1\I^\times(X) =
v$ and $\pi_2\I^\times(X) = v'$ iff $\I^\times(X) = \langle v,
v'\rangle$.

\begin{lemma}\label{lem:tauttimes}
$\Taut(\M \times \M') = \Taut(\M) \cap \Taut(\M')$
\end{lemma}

\begin{proof}
Let $A$ be a tautology of $\M \times \M'$ and $\I$ and $\I'$ be
valuations of $\M$ and $\M'$, respectively.  Since $\I \times \I'
\models_{\M \times \M'} A$, we have $\I \models_\M A$ and $\I'
\models_{\M'} A$ by the definition of~$\times$.  Conversely, let $A$
be a tautology of both \M{} and $\M'$, and let $\I^\times$ be a
valuation of $\M \times \M'$.  Since $\pi_1\I^\times \models_\M A$ and
$\pi_2\I^\times \models_{\M'} A$, it follows that $\I^\times
\models_{\M \times \M'} A$.
\qed\end{proof}

The definition and lemma are easily generalized to the
case of finite products $\prod_i \M_i$ by induction.

The construction of Lindenbaum \cite[Satz~3]{LukasiewiczTarski:30}
shows that every propositional logic can be characterized as the set
of tautologies of an infinite-valued logic.  $\M(\AL)$ is defined as
follows: the set of truth values $V(\M(\AL)) = \Frm(\LA)$, and the set
of designated values $V^+(\M(\AL)) = \AL$.  The truth functions are
given by
\[
\tbox(F_1, \ldots, F_n) = \Box(F_1, \ldots, F_n)
\]
Since we are interested in finite-valued logics, the following
constructions will be useful.

\begin{defn}\label{defn:Mij}

Let $\Frm_{i,j}(\LA)$ be the set of formulas of depth $\le i$
containing only the variables $X_1$, \ldots, $X_j$.  The finite-valued
logic $\M_{i,j}(\AL)$ is defined as follows: The set of truth values
of $\M_{i,j}(\AL)$ is $V = \Frm_{i,j}(\LA) \cup \{\top\}$; the
designated values $V^+ = (\Frm_{i,j}(\LA) \cap \AL) \cup \{\top\}$.
The truth tables for $\M_{i,j}(\AL)$ are given by:
\[
\begin{array}{@{}l@{}}
\tbox(v_1, \ldots, v_n) = \\
\quad = \begin{cases} \Box(F_1, \ldots, F_n) & \text{if $v_j = F_j$ for $1 \le j \le n$} \\
& \text{and $\Box(F_1, \ldots, F_n) \in \Frm_{i,j}(\LA)$} \\
\top & \text{otherwise}
\end{cases}
\end{array}
\]
\end{defn}

\begin{proposition}\label{prop:Mij}
Let $\I$ be a valuation in $\M_{i,j}(\AL)$. If $\I(X) \notin
\Frm_{i,j}(\LA)$ for some $X \in \Var(A)$, then $\I(A) =
\top$. Otherwise, $\I$ can also be seen as a substitution $\sigma_\I$
assigning the formula $\I(X) \in \Frm_{i,j}(\LA)$ to the variable $X$.
Then $\I(A) = A$ if $\depth(A\sigma_\I) \le i$ and $= \top$ otherwise.

If $A \in \Frm_{i,j}(\LA)$, then $A \in \Taut(\M_{i,j}(\AL))$ iff $A
\in \AL$; otherwise $A \in \Taut(\M_{i,j}(\AL))$.  In particular, $\AL
\subseteq \Taut(\M_{i,j}(\AL))$.
\end{proposition}

\begin{proof}
By induction on the depth of $A$.  
\qed\end{proof}

When looking for a logic with as small a number of truth values
as possible which falsifies a given formula we can
use the following construction.

\begin{proposition}\label{prop:red}
Let \M{} be any many-valued logic,
and $A_1$, \dots,~$A_n$ be formulas not valid in~\M.
Then there is a finite-valued logic $\M' = \Phi(\M, A_1, \ldots, A_n)$ s.t.{}
\begin{enumerate}
\item $A_1$, \dots, $A_n$ are not valid in~$\M'$,
\item $\Taut(\M) \subseteq \Taut(\M')$, and
\item $\left| V(\M') \right| \le \xi(A_1, \ldots, A_n)$, where
$\xi(A_1, \ldots, A_n) = \prod_{i=1}^n \xi(A_i)$ and $\xi(A_i)$ is the 
number of subformulas of $A_i\ + 1$. 
\end{enumerate}
\end{proposition}

\begin{proof}
We first prove the proposition for $n = 1$.
Let \I~be the valuation in~\M{} making
$A_1$~false, and let $B_1$, \dots,~$B_r$ ($\xi(A_1) = r+1$)
be all subformulas of~$A_1$.
Every $B_i$ has a truth value~$t_i$ in~\I.
Let $\M'$ be as follows: $V(\M') = \{t_1, \ldots, t_r, \top\}$,
$V^+(\M') = V^+(\M) \cap V(\M') \cup \{\top\}$.
If $\Box \in \LA$, define $\tbox$ by
\[
\tbox(v_1, \ldots, v_n) = \begin{cases}
t_i & \text{if $B_i \equiv \Box(B_{j_1}, \ldots, B_{j_n})$} \\
    & \text{and $v_1 = t_{j_1}$, \dots, $v_n = t_{j_n}$} \\
\top & \text{otherwise} \end{cases}
\]

(1)~Since $t_r$ was undesignated in $\M$, it is also
undesignated in $\M'$. But $\I$ is also
a truth value assignment in $\M'$,
hence $\M' \not\models A_1$.

(2)~Let $C$ be a tautology of \M, and let \J~be a valuation
in~$\M'$.  If no subformula of $C$ evaluates to $\top$
under~\J, then \J{} is also a valuation in~\M, and
$C$ takes the same truth value in $\M'$ as in \M{} w.r.t.~\J,
which is designated also in $\M'$. Otherwise, $C$ evaluates to $\top$,
which is designated in~$\M'$. So $C$ is a tautology in~$\M'$.

(3)~Obvious.

\noindent For $n > 1$, the proposition follows by taking 
$\Phi(\M, A_1, \ldots, A_n) = \prod_{i =1}^n \Phi(\M, A_i)$
\qed\end{proof}

\section{Many-Valued Covers for Propositional Calculi}

A very natural way of representing logics is via calculi.  In the
context of our study, one important question is under what conditions
it is possible to find, given a calculus~\AK, a finite-valued
logic~\M{} which approximates as well as possible the set of theorems
$\Thm(\AK)$.  In the optimal case, of course, we would like to have
$\Taut(\M) = \Thm(\AK)$.  This is, however,not always possible.  In
fact, it is in general not even possible to decide, given a calculus
\AK{} and a finite-valued logic~\M, if \M{} is sound for \AK.  In some
circumstances, however, general results can be obtained.  We begin
with some definitions.

\begin{defn}
A calculus~\AK{} is \emph{weakly sound} for an $m$-valued logic~\M{} provided
$\Thm(\AK) \subseteq \Taut(\M)$.  
\end{defn}

\begin{defn}
A calculus~\AK{} is \emph{t-sound} for an $m$-valued logic~\M{}
if
\begin{enumerate}
\item[($*$)] All axioms $A \in A(\AK)$ are tautologies of~\M, and
for every rule~$r \in R(\AK)$ and substitution $\sigma$: if 
for every premise $A$ of $r$, $A\sigma$ is a tautology, then
the corresponding instance $C\sigma$ of the conclusion of $r$
is a tautology as well.
\end{enumerate}
\end{defn}

\begin{defn}\label{prop:approxtest}
A calculus \AK{} is \emph{strongly sound} for an $m$-valued logic~\M{} if
\begin{enumerate}
\item[($*$$*$)] All axioms $A \in A(\AK)$ are tautologies of~\M, and
for every rule~$r \in R(\AK)$: if a valuation
satisfies the premises of~$r$, it also satisfies the
conclusion.
\end{enumerate}
\M{} is then called a {\em cover} for \AK.
\end{defn}

We would like to stress the distinction between these three notions of
soundness.  soundness.  The notion of weak soundness is the familiar
property of a calculus to produce only valid formulas (in this case:
tautologies of~\M) as theorems.  This ``plain'' soundness is what we
actually would like to investigate in terms of approximations.  More
precisely, when looking for a finite-valued logic that approximates a
given calculus, we are content if we find a logic for which \AK{} is
weakly sound.  This is unfortunately not possible in general.

\begin{proposition}
It is undecidable if a calculus~\AK{} is weakly sound for 
a given $m$-valued logic~\M.
\end{proposition}

\begin{proof}
Let \AK{} be an undecidable propositional calculus, let $F$ be a
formula, and let $C$ and, for each $X_i \in \Var(F)$, $C_i$ be new
propositional constants ($0$-ary connective) not occurring in~\AK. Let
$\sigma\colon X_i \mapsto C_i$ be a substitution.  Clearly, $\AK
\vdash F$ iff $\AK \vdash F\sigma$.  Now let $\AK'$ be~\AK{} with the
additional rule
\[
\infer{C}{F\sigma}
\]
and let $\M$ be an $m$-valued logic which assigns a non-designated
value to $C$ and otherwise interprets every connective as a constant
function with a designated value.  Then every formula except a
variable of the original language is a tautology, and $C$ is not.
\M{} is then weakly sound for \AK{} over the original language, but
weakly sound for $\AK'$ iff $C$ is not derivable.  Moreover, $\AK'
\vdash C$ iff $\AK' \vdash F\sigma$, i.e., iff $\AK \vdash F$.  If it
were decidable whether \M{} is weakly sound for $\AK'$ it would then
also be decidable if $\AK \vdash F$, contrary to the assumption
that~\AK{} is undecidable.
\qed\end{proof}

On the other hand, it {\em is} obviously decidable if \AK{} is strongly sound
for a given matrix~\M.  

\begin{proposition}\label{cor:dec}
It is decidable if a given propositional calculus is
strongly sound for a given
$m$-valued logic.
\end{proposition}

\begin{proof}
($*$$*$) can be tested by the usual truth-table method.
\qed\end{proof}

It is also decidable if~\AK{} is t-sound for a matrix~\M, although
this is less obvious:

\begin{proposition}
It is decidable if a given propositional calculus
is t-sound for a given $m$-valued logic~\M.
\end{proposition}

\begin{proof}
Let $r$ be a rule with premises $A_1$, \dots, $A_n$ and conclusion $C$
containing the variables $X_1$, \dots, $X_k$, and $\sigma$ a
substitution.  If $A_1\sigma$, \dots, $A_n\sigma$ are tautologies, but
$C$ is not, then ($*$) is violated and $r$ is not weakly sound.  Given
$\sigma$, this is clearly decidable.  We have to show that there are
only a finite number of substitutions $\sigma$ which we have to test.

Let $Y_1$, \dots, $Y_l$ be the variables occurring in $X_1\sigma$,
\dots, $X_k\sigma$.  We show first that it suffices to consider
$\sigma$ with $l = m$.  For if $\I$ is a valuation in which $C\sigma$
is false, then at most $m$ of $Y_1$, \dots, $Y_l$ have different truth
values.  Let $\tau$ be a substitution so that $\tau(Y_i) = Y_j$ where
$j$ is the least index such that $\I(Y_j) = \I(Y_i)$.  Then (1)
$\I(C\sigma\tau) = \I(C\sigma)$ and hence $C\sigma\tau$ is not a
tautology; (2) $A_i\sigma\tau$ is still a tautology; (3) there are at
most $m$ distinct variables occurring in $A_1\sigma\tau$, \dots,
$A_n\sigma\tau$, $C\sigma\tau$.

Now every $B_i = X_i\sigma$ defines an $m$-valued function of $m$
arguments.  There are $m^{m^m}$ such functions.  Whether $A_i\sigma$
is a tautology only depends on the function defined by $B_i$, but it
is not prima facie clear which functions can be expressed in~\M.
Nevertheless, we can give a bound on the depth of formulas $B_i$ that
have to be considered.  Suppose $\sigma$ is a substitution of the
required form with $B_i=X_i\sigma$ of minimal depth and suppose that
the depth of $B_i$ is greater than $m' = m^{m^m}$. Now consider a
sequence of formulas $C_1$, \dots, $C_{m'+1}$ with $C_1 = B_i$ and
each $C_j$ a subformula of $C_{j-1}$.  Each $C_j$ also expresses an
$m$-valued function of $m$ arguments.  Since there are only $m'$
different such functions, there are $j < j'$ so that $C_j$ and
$C_{j'}$ define the same function.  The formula obtained from $B_i$ by
replacing $C_j$ by $C_{j'}$ expresses the same function.  Since this
can be done for every sequence of $C_j$'s of length $> m'$ we
eventually obtain a formula which expresses the same function as $B_i$
but of depth $\le m'$, contrary to the assumption that it was of
minimal depth.
\qed\end{proof}

Now, if \AK{} is strongly sound for \M, it is also t-sound; and
if it is t-sound, it is also weakly sound.  The converses, however,
are false:

\begin{example}\label{ex:soundness}
Let $\cal L$ be the language consisting of a unary connective $\Box$
and a binary connective $\vartriangleleft$, and let \AK{} be the
calculus consisting of the sole axiom $X \vartriangleleft \Box X$ and
the rules
\[
\infer[r_1]{X \vartriangleleft Z}{X \vartriangleleft Y & Y \vartriangleleft Z} \qquad
\infer[r_2]{Y}{X \vartriangleleft X}
\] 
It is easy to see that the only derivable formulas in~\AK{} using only
rule $r_1$ are substitution instances of $\Box^\ell X \vartriangleleft
\Box^k X$ with $\ell < k$.  In particular, no substitution instance of
the premise of $r_2$, $X \vartriangleleft X$, is derivable.  It
follows that rule $r_2$ can never be applied.  We now show that
if~\AK{} is strongly sound for an $m$-valued matrix~\M, $\Taut(\M) =
\Frm(\LA)$, i.e., \M{} is trivial.  Suppose \M{} is given by the set
of truth values $V = \{1, \ldots, m\}$.  Since $X \vartriangleleft
\Box X$ must be a tautology, $\widetilde\vartriangleleft(i,
\widetilde\Box(i)) \in V^+$ for $i = 1$, \dots, $m$. Since \AK{} is
strongly sound for rule $r_1$, and by induction,
$\widetilde\vartriangleleft(i, \widetilde\Box^k(i)) \in V^+$ for
all~$k$.  Since $V$ is finite, there are $i$ and $k$ such that $i =
\widetilde\Box^k(i)$.  Then $\widetilde\vartriangleleft(i, i) \in
V^+$. Since \AK{} is strongly sound for $r_2$, we have $V = V^+$.
However, \AK{} is weakly sound for non-trivial matrices, e.g., $\M'$
with $V' = \{1, \ldots, k\}$, $V^+ = \{k\}$, $\widetilde\Box(i) = i +
1$ for $i < k$ and $= k$ otherwise, and $\widetilde\vartriangleleft(i,
j) = k$ if $i < j$ or $j = k$ and $= 1$ otherwise.  \AK{} is, however,
also not t-sound for this matrix.
\end{example}

\begin{example}
Consider the calculus with propositional constants $T$, $F$, and
binary connective $\neq$, the axiom $T \neq F$ and the rules
\[
\infer[r_1]{X \neq Y}{Y \neq X} \qquad \infer[r_2]{Y}{X \neq T & X \neq F}
\]  
and the matrix with $V = \{0, 1, 2\}$, $V^+ = \{2\}$, $\widetilde T =
2$, $\widetilde F = 0$, and $\widetilde\neq(i, j) = 2$ if $i \neq j$
and $=0$ otherwise.  Clearly, the only derivable formulas are $T \neq
F$ and $F \neq T$, which are also tautologies. The calculus is not
strongly sound, since for $\I(X) = 1$, $\I(Y) = 0$ the premises of
$r_2$ are designated, but the conclusion is not.  It is, however,
t-sound: only a substitution $\sigma$ with $\I(X\sigma) = 1$ for all
valuations~\I{} would turn both premises of $r_2$ into tautologies,
and there can be no such formulas.  Hence, we have an example of a
calculus t-sound but not strongly sound for a matrix.
\end{example}

\begin{example}\label{ex:goedelsound}
The \IPC{} is strongly sound for the $m$-valued G\"odel logics~\Gm{}.
For instance, take axiom $a_5$:  $B \impl (A \impl B)$.  This is a
tautology in \Gm, for assume we assign some truth values $a$ and $b$
to $A$ and $B$, respectively.  We have two cases:  If $a \le b$, then
$(A \impl B)$ takes the value~$m-1$.  Whatever $b$ is, it certainly is
$\le m-1$, hence $B \impl A \impl B$ takes the designated value $m-1$.
Otherwise, $A \impl B$ takes the value $b$, and again (since $b \le
b$), $B \impl A \impl B$ takes the value $m-1$.

Modus ponens passes the test: Assume $A$ and $A \impl B$ both
take the value $m-1$. This means that $a \le b$. But $a = m-1$, hence
$b = m-1$.

Now consider the following extension $\G_m^\top$ of \Gm:
$V(\G_m^\top) = V(\Gm) \cup \{\top\}$, $V^+(\G_m^\top) = \{m -1, \top\}$,
and the truth functions are given by:
\[
\tbox_{\G_m^\top}(\bar{v}) = \begin{cases}\top & \text{if $\top \in \bar{v}$}\\
\tbox_{\Gm}(\bar{v}) & \text{otherwise} \end{cases}
\] 
for $\Box \in \{\neg, \impl, \land, \lor\}$.  \IPC{} is not strongly
sound for $\G_m^\top$, since a valuation with $\I(X) = \top$, $\I(Y) =
0$ would satisfy the premises of rule MP, $X$ and $X \impl Y$, but not
the conclusion $Y$.  However, a calculus in which the conclusion of
each rule contains all variables occurring in the premises, is
strongly sound (such as a calculus obtained from \textbf{LJ} using the
construction outlined in Remark~\ref{rem:sequent}).
\end{example}

\begin{example}
Consider the following calculus~{\bf K}:
\[
X \cbi \Nex X \qquad \infer[r_1]{X \cbi \Nex Y}{X \cbi Y}
\qquad \infer[r_2]{Y}{X \cbi X}
\]
It is easy to see that the corresponding logic consists
of all instances of $X \cbi \Nex^k X$ where $k \ge 1$.
This calculus is only strongly sound for the $m$-valued
logic having all formulas as its tautologies.
But if we leave out $r_2$, we can give a sequence of many-valued
logics $\M_i$, for each of which {\bf K} is strongly sound:
Take for $V(\M_n) = \{0, \ldots, n-1\}$, $V^+(\M_n) = \{0\}$,
with the following truth functions:
\begin{eqnarray*}
\widetilde{\Nex}v & = & \begin{cases} v+1 & \text{if $v < n-1$} \\
n-1 & \text{otherwise} \end{cases}\\
v \widetilde{\cbi} w & = & \begin{cases}0 & \text{if $v < w$ or $v = n-1$}\\
1 & \text{otherwise} \end{cases}
\end{eqnarray*}
Obviously, $\M_n$ is a cover for {\bf K}.
On the other hand, $\Taut(\M_n) \neq \Frm(\LA)$, e.g., any formula
of the form $\Nex(A)$ takes a (non-designated) value~$>0$ (for $n > 1$).
In fact, every formula of the form $\Nex^k X \cbi X$
is falsified in some~$\M_n$. 
\end{example}

\section{Optimal Covers}

By Proposition~\ref{cor:dec} it is decidable if a given $m$-valued
logic~\M{} is a cover of~\AK.  Since we can enumerate all $m$-valued
logics, we can also find all covers of~\AK.  Moreover, comparing two
many-valued logics as to their sets of tautologies is decidable, as
the next theorem will show.  Using this result, we see that we can
always generate optimal covers for \AK{}.

\begin{defn}
For two many-valued logics $\M_1$ and $\M_2$, we write
$\M_1 \bettereq \M_2$ iff $\Taut(\M_1) \subseteq \Taut(\M_2)$.

$\M_1$ is {\em better} than~$\M_2$, $\M_1 \better
\M_2$, iff $\M_1 \bettereq \M_2$ and $\Taut(\M_1) \neq \Taut(\M_2)$.
\end{defn}

\begin{theorem}
Let two logics $\M_1$ and $\M_2$, $m_1$-valued and $m_2$-valued respectively,
be given. It is decidable whether $\M_1 \better \M_2$.
\end{theorem}

\begin{proof}
It suffices to show the decidability of the following property:  There
is a formula~$A$, s.t.\ (*)~$\M_2 \models A$ but $\M_1 \not\models A$.
If this is the case, write $\M_1 \better^* \M_2$.
$\M_1 \better \M_2$ iff $\M_1 \better^* \M_2$ and not
$\M_2 \better^* \M_1$.

We show this by giving an upper bound on the depth of a minimal
formula~$A$ satisfying the above property.  Since the set of formulas
of~$\LA$ is enumerable, bounded search will produce such a
formula iff it exists.  Note that the property~(*) is decidable by
enumerating all assignments. In the following, let $m = \max(m_1,m_2)$.

Let~$A$ be a formula that satisfies~(*), i.e., there is a valuation
\I{} s.t.\ $\I \not\models_{\M_1} A$.  W.l.o.g.\ we can assume that
$A$~contains at most~$m$ different variables:  if it contained more,
some of them must be evaluated to the same truth value in the
counterexample~\I{} for $\M_1 \not\models A$.  Unifying these
variables leaves~(*) intact.

Let $B = \{B_1, B_2, \ldots \}$ be the set of all subformulas of~$A$.
Every formula~$B_j$ defines an $m$-valued truth function~$f(B_j)$ of
$m$~variables where the values of the variables which actually occur
in~$B_j$ determine the value of~$f(B_j)$ via the matrix of~$\M_2$.  On
the other hand, every~$B_j$ evaluates to a single truth value~$t(B_j)$
in the countermodel~\I.

Consider the formula~$A'$ constructed from $A$ as follows:  Let~$B_i$
be a subformula of~$A$ and $B_j$~be a proper subformula of~$B_i$ (and
hence, a proper subformula of~$A$).  If $f(B_i) = f(B_j)$ and $t(B_i)
= t(B_j)$, replace $B_i$ in $A$ with $B_j$.  $A'$ is shorter than~$A$,
and it still satisfies~(*).  By iterating this construction
until no two subformulas have the desired property we obtain a
formula~$A^*$.  This procedure terminates, since $A'$ is shorter than
$A$; it preserves~(*), since $A'$ remains a tautology
under~$\M_2$ (we replace subformulas behaving in exactly the same way
under all valuations) and the countermodel~\I{} is also a countermodel
for~$A'$.

The depth of $A^*$ is bounded above by $m^{m^m+1}-1$.  This is seen as
follows:  If the depth of $A^*$ is~$d$, then there is a sequence $A^*
= B_0', B_1', \ldots, B_d'$ of subformulas of $A^*$ where $B_k'$ is an
immediate subformula of $B_{k-1}'$.  Every such $B_k'$ defines a truth
function~$f(B_k')$ of $m$ variables in $\M_2$ and a truth valued
$t(B_k')$ in $\M_1$ via~\I.  There are $m^{m^m}$ $m$-ary truth
functions of $m$ truth values.  The number of distinct truth
function-truth value pairs then is $m^{m^m+1}$.  If $d \ge m^{m^m+1}$,
then two of the $B_k'$, say $B_i'$ and $B_j'$ where $B_j'$ is a
subformula of $B_i'$ define the same truth function and the same truth
value.  But then $B_i'$ could be replaced by $B_j'$, contradicting the
way $A^*$ is defined.
\qed\end{proof}

\begin{corollary}
It is decidable if two many-valued logics define the same
set of tautologies. The relation $\bettereq$ is decidable.
\end{corollary}

\begin{proof}
$\Taut(\M_1) = \Taut(\M_2)$ iff neither $\M_1 \better^* \M_2$ nor
$\M_2 \better^* \M_1$.
\qed\end{proof}

Let $\simeq$ be the equivalence relation on $m$-valued logics defined
by: $\M_1 \simeq \M_2$ iff $\Taut(M_1) = \Taut(M_2)$, and let $\MVL_m$
be the set of all $m$-valued logics over~\LA with truth value set
$\{1, \ldots, m\}$.  By~${\cal M}_m$ we denote the set of all
sets~$\Taut(\M)$ of tautologies of $m$-valued logics~$\M$.  The
partial order $\langle {\cal M}_m, \subseteq\rangle$ is isomorphic
to~$\langle \MVL_m/\simeq, \bettereq/\simeq\rangle$.

\begin{proposition}
The optimal (i.e., minimal under
$\better$) $m$-valued covers of \AK{} are computable.
\end{proposition}

\begin{proof}
Consider the set~$C_m(\AK)$
of $m$-valued covers of~\AK. Since $C_m(\AK)$ is finite
and partially ordered by~$\bettereq$, $C_m(\AK)$ contains minimal
elements. The relation $\bettereq$ is decidable, hence the
minimal covers can be computed.
\qed\end{proof}

\begin{example}
By Example~\ref{ex:goedelsound},
\IPC{} is strongly sound for $\G_3$.
The best 3-valued approximation of~\IPC{}
is the 3-valued G\"odel logic. In fact, it is the only 3-valued
approximation of {\em any} sound calculus~\AK{} (containing modus ponens)
for \IPL{} which has less tautologies than classical logic~\CL.
This can be seen as follows: Consider the fragment containing
$\bot$ and $\impl$ ($\neg B$ is usually defined as $B \impl \bot$).
Let \M{} be some 3-valued strongly sound approximation of~\AK.
By G\"odel's double-negation translation, 
$B$ is a classical tautology iff $\neg\neg B$ is true intuitionistically.
Hence, whenever $\M \models \neg \neg X \impl X$, then
$\Taut(\M) \supseteq \CL$. Let $0$ denote the value of $\bot$ in \M,
and let $1 \in V^+(\M)$. We distinguish cases:
\begin{enumerate}
\item $0 \in V^+(\M)$: Then $\Taut(\M) = \Frm(\LA)$, since $\bot \impl X$
is true intuitionistically, and by modus ponens: $\bot, \bot \impl X / X$.
\item $0 \notin V^+(\M)$: Let $u$ be the third truth value.
\begin{enumerate}
\item $u \in V^+(\M)$: Consider $A \equiv ((X \impl \bot) \impl \bot) \impl X$.
If $\I(X)$ is $u$~or~$1$, then, since everything implies something true,
$A$ is true (Note that we have $Y, Y \impl (X \impl Y) \vdash X \impl Y$).
If $\I(X) = 0$, then (since $0 \impl 0$ is true, but $u \impl 0$ and $1 \impl 0$
are both false), $A$ is true as well. So $\Taut(\M) \supseteq \CL$.
\item $u \notin V^+(\M)$, i.e., $V^+(\M) = \{1\}$: Consider the
truth table for implication.  Since $B \impl B$, $\bot \impl B$,
and something true is implied by everything, the upper right triangle
is~$1$. We have the following table:
\[
\begin{array}{c|ccc}
\impl & 0 & u & 1 \\ \hline
0 &     1 & 1 & 1 \\
u &     v_1 & 1 & 1 \\
1 &     v_0 & v_2 & 1
\end{array} 
\]
Clearly, $v_0$ cannot be~$1$. If $v_0 = u$, we have, by $((X \impl X) \impl \bot) \impl Y$,
that $v_1 = 1$. In this case, $\M \models A$ and hence $\Taut(\M) \supseteq \CL$.
So assume $v_0 = 0$.
\begin{enumerate}
\item $v_1 = 1$: $\M \models A$ (Note that only the case of $((u \impl 0) \impl 0) \impl u$
has to be checked).
\item $v_1 = u$: $\M \models A$.
\item $v_1 = 0$: With $v_2 = 0$, \M{} would be incorrect ($u \impl (1 \impl u)$ is false).
If $v_2 = 1$, again $\M \models A$. The case of $v_2 = u$ is the G\"odel logic,
where $A$ is not a tautology.  
\end{enumerate}
\end{enumerate}
\end{enumerate}
\end{example}

Note that it is in general impossible to effectively construct a
$\bettereq$-minimal $m$-valued logic~\M{} with $\AL \subseteq
\Taut(\M)$ if \AL{} is given independently of a calculus, because,
e.g., it is undecidable whether \AL{} is empty or not: e.g., take
\[\AL = \begin{cases} \{\Box^k(\top)\} & \text{if $k$ is the least
solution of $D(x) = 0$}\\ \emptyset & \text{otherwise}\end{cases}\]
where $D(x) = 0$ is the Diophantine representation of some undecidable
set.

\section{Effective Sequential Approximations}

In the previous section we have shown that it is always possible to
obtain the best $m$-valued covers of a given calculus, but there is no
way to tell {\em how good} these covers are.  In this section, we
investigate the relation between sequences of many-valued logics and
the set of theorems of a calculus~\AK.  Such sequences are called {\em
sequential approximations} of \AK{} if they verify all theorems and
refute all non-theorems of~\AK, and \emph{effective} sequential
approximations if they are effectively enumerable.  This
is also a question about the limitations of Bernays's method.  On the
negative side, an immediate result says that calculi for undecidable
logics do not have effective sequential approximations.  If, however, a
propositional logic is decidable, it also has an effective sequential
approximation (independent of a calculus).  Moreover, any calculus has a
uniquely defined {\em many-valued closure}, whether it is decidable
or not.  This is the set of all sentences which cannot be proved
underivable using a Bernays-style many-valued argument.  If a calculus
has an effective sequential approximation, then the set of its theorems equals
its many-valued closure.  If it does not, then its closure is a proper
superset.  Different calculi for one and the same logic may have
different many-valued closures according to their degree of
analyticity.

\begin{defn}\label{defn:sapprox}
Let \AL{} be a propositional logic and let $\A = \langle \M_1, \M_2,
\M_3, \ldots, \rangle$ be a sequence of
many-valued logics s.t.  (1) $\M_i \bettereq \M_j$ iff $i \ge j$.

\A{} is called a {\em sequential approximation} of \AL{} iff
$\AL = \bigcap_{j \in \omega} \Taut(\M_j)$.  If in addition \A{} is effectively
enumerated, then \A{} is an \emph{effective sequential approximation}.

If \AL{} is given by the calculus \AK, and each $M_j$ is a cover of \AK,
then \A{} is a \emph{strong (effective) sequential approximation} of \AK{}
(if \A{} is effectively enumerated).
 
We say \AK{} is {\em effectively approximable}, if there is
such a strong effective sequential approximation of~\AK.
\end{defn}

Condition (1) above is technically not necessary.  Approximating sequences
of logics in the literature (see next example), however, satisfy this
condition.  Furthermore, with the emphasis on ``approximation,'' it seems more
natural that the sequence gets successively ``better.''

\begin{example}
Consider the sequence $\G = \langle \G_i \rangle_{i\ge 2}$ of G\"odel
logics and intuitionistic propositional logic \IPC.  $\Taut(\G_i)
\supset \Thm(\IPC)$, since $\G_i$ is a cover for~\IPC. Furthermore,
$\G_{i+1} \better \G_i$. This has been pointed out by \cite{Godel:32},
for a detailed proof see \cite[Theorem 10.1.2]{Gottwald:01}.  It is,
however, not a sequential approximation of \IPC: The formula $(A \impl
B) \lor (B \impl A)$, while not a theorem of \IPL, is a tautology of
all $\G_i$. In fact, $\bigcap_{j \ge 2} \Taut(\G_i)$ is the set of
tautologies of the infinite-valued G\"odel logic~$\G_\aleph$, which is
axiomatized by the rules of \IPC{} plus the above formula.  This has
been shown in \cite{Dummett:59} (see also \cite[Section 10.1]{Gottwald:01}).
Hence, \G{} is a strong effective sequential approximation of
$\G_\aleph = \IPC + (A \impl B) \lor (B \impl A)$.

Ja{\'s}kowski \cite{Jaskowski:36} gave an effective strong sequential
approximation of \IPC.  That \IPC~is approximable is also a
consequence of Theorem~\ref{thm:fmpapprox}, with the proof adapted to
Kripke semantics for intuitionistic propositional logic, since
\IPL~has the finite model property \cite[Ch.~4,
Theorem~4(a)]{Gabbay:81}.
\end{example}

The natural question to ask is: Which logics have (effective)
sequential approximations; which calculi are approximable?  

First of all, any propositional logic has a sequential approximation,
although it need not have an effective approximation.

\begin{proposition}
Every propositional logic~\AL{} has a sequential approximation.
\end{proposition}

\begin{proof}
A sequential approximation of~\AL{} a is given by $\M_i =
\M_{i,i}(\AL)$ (see Definition~\ref{defn:Mij}).  Any formula $F \notin \AL$
is in $V(\M_k)$ for $k = \max\{\depth(F), j\}$ where $j$ is the
maximum index of variables occurring in~$F$.  By
Proposition~\ref{prop:Mij}, $F$ is falsified in $\M_k$. Also,
$\Taut(\M_i) \supseteq \AL$, and $\M_i \bettereq \M_{i+1}$.
\qed\end{proof}

\begin{corollary}
If \AL{} is decidable, it has an effective sequential approximation.
\end{corollary}

\begin{proof}
Using a decision procedure for \AL{}, we can effectively enumerate the
$\M_{i,i}(\AL)$.
\qed\end{proof}

\begin{proposition}
If \AL{} has an effectively sequential approximation, then $\Frm(\LA)
\setminus \AL$ is effectively enumerable.
\end{proposition}

\begin{proof}
Suppose there is an effectively enumerated sequence $\A = \langle
\M_1, \M_2, \ldots\rangle$ s.t. $\bigcap_{j\ge 2} \Taut(\M_j) =
\AL$.  If $F \notin \AL$ then there would be an index~$i$ s.t.\ $F$ is
false in $\M_i$.  But this would yield a semi-decision procedure for
non-members of~\AL: Try for each~$j$ whether $F$~is false in
$\M_j$. If $F \notin \AL$, this will be established at $j = i$.
\qed\end{proof}

\begin{corollary}\label{prop:undecapprox}
If \AK{} is undecidable, then it is not effectively approximable.
\end{corollary}

\begin{proof}
$\Thm(\AK)$ is effectively enumerable.  If \AK{} were approximable, it
would have an effective sequential approximation, and this contradicts
the assumption that the non-theorems of \AK{} are not effectively
enumerable. 
\qed\end{proof}

\begin{example}
This shows that a result similar to that for~\IPC{}
cannot be obtained for full propositional linear logic.
\end{example}

If \AK{} is not effectively approximable (e.g., if it is undecidable),
then the intersection of all covers for~\AK{} is a proper superset
of~$\Thm(\AK)$.  This intersection has interesting properties.

\begin{defn}
The {\em many-valued closure}~$\MC(\AK)$ of a calculus~\AK{}
is the set of formulas which are true in every many-valued
cover for~\AK{}.
\end{defn}

\begin{proposition}
$\MC(\AK)$ is unique and has an effective sequential approximation.
\end{proposition}

\begin{proof}
$\MC(\AK)$ is unique, since it obviously equals $\bigcap_{\M \in S}
\Taut(\M)$ where $S$ is the set of all covers for~\AK.  It is also
effectively approximable, an approximating sequence is given by
\begin{eqnarray*}
\M_1 & = & \M'_1 \\
\M_i & = & \M_{i-1} \times \M_i'
\end{eqnarray*}
where $\M_i'$ is an effective enumeration of~$S$.
\qed\end{proof}

Since $\MC(\AK)$ is defined via the many-valued logics for which $\AK$
is strongly sound, it need not be the case that $\MC(\AK) = \Thm(\AK)$
even if $\AK$ is decidable.  (An example is given below.) On the other
hand, it also need not be trivial (i.e., equal to $\Frm(\LA)$) even
for undecidable~\AK{}.  For instance, take the Hilbert-style calculus
for linear logic given in \cite{Avron:88,Troelstra:92}, and the
2-valued logic which interprets the linear connectives classically and
the exponentials as the identity.  All axioms are then tautologies and
the rules (modus ponens, adjunction) are strongly sound, but the
matrix is clearly non-trivial.

\begin{corollary}
If \AK{} is strictly analytic, then $\MC(\AK) = \Thm(\AK)$.
\end{corollary}

\begin{proof}
We have to show that for every $F \notin \Thm(\AK)$ there is a
finite-valued logic $\M$ which is strongly sound for $\AK$ and where
$F \notin \Taut(\M)$. Let $X_1$, \ldots, $X_j$ be all the variables
occurring in $F$ and the axioms and rules of~\AK. 
Then set $\M = \M_{\depth(F),j}(\Thm(\AK))$.

By Proposition~\ref{prop:Mij}, $F \notin \Taut(\M)$ and all axioms of
\AK{} are in $\Taut(\M)$.  Now consider a valuation $v$ in \M{} and
suppose $v(A_i) \in V^+$ for all premises $A_i$ of a rule of~\AK.  We
have two cases: if $v(X) = \top$ for some variable $X$ appearing in a
premise~$A_i$, then, since \AK{} is strictly analytic, $X$ also appears in the
conclusion $C$ and hence $v(C) = \top$.  Otherwise, let $\sigma_v$ be
the substitution corresponding to $v$.  If $v(A_i) = \top$ for
some~$i$, this means that $\depth(A_i\sigma) > \depth(F)$.  By strict analyticity,
$\depth(C\sigma) \ge \depth(A_i\sigma) > \depth(F)$ and hence $v(C) = \top$.
Otherwise, $v(A_i) = A_i\sigma$ for all premises~$A_i$. Since $v(A_i)$
is designated, $A_i\sigma \in \Thm(\AK)$, hence $C\sigma \in \Thm(\AK)$.
Then either $v(C) = C\sigma$ or $v(C) = \top$, and both are in~$V^+$.
\qed\end{proof}

\begin{example}
The last corollary can be used to uniformly obtain semantics for
strictly analytic Hilbert calculi.  Strict analyticity of the
calculus is a necessary condition, as Example~\ref{ex:soundness}
shows.  The calculus given there is decidable, though not strictly analytic,
and has only trivial covers.  Its set of theorems nevertheless has an
effective sequential approximation, i.e., it is the intersection of an
infinite sequence of finite-valued matrices which are weakly sound
for~\AK.  For this it is sufficient to give, for each formula $A$
s.t. $\AK \nvdash A$, a matrix~\M{} weakly sound for \AK{} with $A
\notin \Taut(\M)$.  Let the depth of $A$ be $k$, and let
\begin{eqnarray*}
V_0 & = & \Var(A) \cup \{\dagger\} \\ V_{i+1} & = & V_i \cup \{B
\vartriangleleft C \mid B, C \in V_i\} \cup \{\Box B \mid B \in V_i\}
\end{eqnarray*}
Then set $V = V_k$, $V^+ = \{B \vartriangleleft C \mid B
\vartriangleleft C \in V, C \equiv \Box^l B\} \cup \{\dagger\}$.
The truth functions are defined as follows:
\begin{eqnarray*}
\widetilde\Box(B) & = & \begin{cases} 
                            \dagger & \text{if $B \in V_k$ but $B \notin V_{k-1}$, or $B = \dagger$} \\
                            \Box B & \text{otherwise}
			\end{cases} \\
\widetilde\vartriangleleft(B, C) & = & 
     \begin{cases}
       \dagger & \text{if $C \in V_k$ but $C \notin V_{k-1}$, or $B = \dagger$} \\
       \dagger \vartriangleleft C & \text{else if $B \in V_k$ but $B \notin V_{k-1}$} \\
	 B \vartriangleleft C & \text{otherwise}
     \end{cases}
\end{eqnarray*}
The axiom $X \vartriangleleft \Box X$ is a tautology. For if $\I(X) =
\dagger$, then $\I(\Box X) = \dagger$ and hence $\I(X \vartriangleleft
\Box X) = \dagger \in V^+$.  If $\I(X) \in V_k$ but $\notin V_{k-1}$,
then $\I(\Box X) = \dagger$ and $\I(X \vartriangleleft \Box X) =
\dagger$.  Otherwise $\I(\Box X) = \Box B$ for $B = \I(A)$ and $\I(X
\vartriangleleft \Box X) = \dagger$ (if $\Box B \in V_k$) or $= B
\vartriangleleft \Box B \in V^+$.

If $\I(X) = \dagger$, then $\I(X \vartriangleleft Z) = \dagger$.
Otherwise $\I(X \vartriangleleft Z) \in V^+$ only if $\I(X) = B$ and
$\I(Y) = \Box^lB$ and $B \notin V_k$. Then, in order for $\I(Y
\vartriangleleft Z)$ to be $\in V^+$, either $\I(Y \vartriangleleft Z)
= \dagger$, in which case $\I(Z) \in V_k$ but $\notin V_{k-1}$, and
hence $\I(Y \vartriangleleft Z) = \dagger$, or $\I(Y \vartriangleleft
Z) = \Box^l B \vartriangleleft \Box^{l'} B$ with $l < l'$, in which
case $\I(Y \vartriangleleft Z) = B \vartriangleleft \Box^{l'} B$.

However, $A \notin Taut(\M)$.  For it is easy to see by induction that
in the valuation with $\I(X) = X$ for all variables $X \in \Var(A)$,
$\I(B) = B$ as long as $B \in V_k$ and so $\I(A)$ can only be
designated if $A \equiv B \vartriangleleft \Box B$ for some $B$, but
all such formulas are derivable in~\AK.

However, there are substitution instances of $X \vartriangleleft X$,
viz., for any $\sigma$ with $X\sigma$ of depth $> k$, for which $(X
\vartriangleleft X)\sigma$ is a tautology. Even though \AK{} is weakly
sound for \M, it is not t-sound.
\end{example}

So far we have concentrated on approximations of logics given via
calculi. However, propositional logics are also often defined via
their semantics.  The most important example of such logics are modal
logics, where logics can be characterized using families of Kripke
structures.  If these Kripke structures satisfy certain properties,
they also yield sequential approximations of the corresponding logics.
Unsurprisingly, for this it is necessary that the modal logics have
the \emph{finite model property}, i.e., they can be characterized by a
family of finite Kripke structures.  The sequential approximations
obtained by our method are only effective, however, if the Kripke
structures are effectively enumerable.

\begin{defn}
A {\em modal logic}~\AL{} has as its language~$\LA$
the usual propositional connectives plus two unary
{\em modal operators:} $\Box$ (necessary) and $\Diamond$ (possible).
A {\em Kripke model} for~$\cal L$ is a triple $\langle W, R, P\rangle$,
where
\begin{enumerate}
\item $W$ is any set: the set of {\em worlds},
\item $R \subseteq W^2$ is a binary relation on $W$: the {\em accessibility relation},
\item $P$ is a mapping from the propositional variables to subsets of~$W$.
\end{enumerate}
A modal logic \AL{} is characterized by a class of Kripke models for \AL.
\end{defn}

This is called the {\em standard semantics} for modal logics
(see \cite[Ch.~3]{Chellas:80}). The semantics of formulas in standard
models is defined as follows:

\begin{defn}
Let \AL{} be a modal logic, $\K_\AL$ be its characterizing class of Kripke
models. Let $K = \langle W, R, P\rangle \in \K_\AL$ be a Kripke model
and $A$ be a modal formula.

If $\alpha \in W$ is a possible world, then we say $A$ is {\em true in $\alpha$},
$\alpha \models_\AL A$, iff the following holds:
\begin{enumerate}
\item $A$ is a variable: $\alpha \in P(X)$
\item $A \equiv \neg B$: not $\alpha \models_\AL B$
\item $A \equiv B \land C$: $\alpha \models_\AL B$ and $\alpha \models_\AL C$
\item $A \equiv B \lor C$: $\alpha \models_\AL B$ or $\alpha \models_\AL C$
\item $A \equiv \Box B$: for all $\beta \in W$ s.t. $\alpha \mathrel{R} \beta$ it holds that $\beta \models_\AL B$
\item $A \equiv \Diamond B$: there is a $\beta \in W$ s.t. $\alpha \mathrel{R} \beta$ and $\beta \models_\AL B$
\end{enumerate}
We say $A$ is {\em true} in $K$, $K \models_\AL A$, iff for all $\alpha \in W$
we have $\alpha \models_\AL A$. $A$ is {\em valid in \AL}, $\AL \models A$,
iff $A$ is true in every Kripke model $K \in \K_\AL$. By $\Taut(\AL)$ we
denote the set of all formulas valid in \AL.
\end{defn}

Many of the modal logics in the literature have the {\em finite model
property (fmp)}: for every $A$ s.t.\ $\AL \not\models A$, there is a
finite Kripke model $K = \langle W, R, P\rangle \in \K$ (i.e., $W$ is
finite), s.t.\ $K \not\models_\AL A$ (where \AL{} is characterized
by~$\K$).  We would like to exploit the fmp to construct sequential
approximations. This can be done as follows:

\begin{defn}\label{defn:MA}
Let $K = \langle W, R, P\rangle$ be a finite Kripke model.  We define
the many-valued logic $\M_K$ as follows:

\begin{enumerate}

\item $V(\M_K) = \{0, 1\}^W$, the set of 0-1-sequences with indices
from $W$.

\item $V^+(\M_K) = \{1\}^W$, the singleton of the sequence constantly
equal to~1.

\item $\widetilde\neg_{\M_K}$, $\widetilde\lor_{\M_K}$,
$\widetilde\land_{\M_K}$, $\widetilde\impl_{\M_K}$ are defined
componentwise from the classical truth functions

\item $\tbox_{\M_K}$ is defined as follows:
\[
\tbox_{\M_K}(\langle w_\alpha \rangle_{\alpha \in W})_\beta =
\begin{cases}1 & \text{if for all $\gamma$ s.t. $\beta \mathrel{R} \gamma$, $w_\gamma = 1$} \\
0 & \text{otherwise}
\end{cases}
\]
\item $\tdiamond_{\M_K}$ is defined as follows:
\[
\tdiamond_{\M_K}(\langle w_\alpha \rangle_{\alpha \in W})_\beta =
\begin{cases}1 & \text{if there is a $\gamma$ s.t. $\beta \mathrel{R} \gamma$ and $w_\gamma = 1$} \\
0 & \text{otherwise} \end{cases}
\]
\end{enumerate}
Furthermore, $\I_K$ is the valuation defined by
$\I_K(X)_\alpha = 1$ iff $\alpha \in P(X)$ and $= 0$ otherwise.
\end{defn}

\begin{lemma}\label{lem:MKmodels}
Let $\AL$ and $K$ be as in Definition~\ref{defn:MA}.
Then the following hold:
\begin{enumerate}
\item Every valid formula of \AL{} is a tautology of $\M_K$.
\item If $K \not\models_\AL A$ then $\I_K \not\models_{\M_K} A$.
\end{enumerate}
\end{lemma}

\begin{proof}
Let $B$ be a modal formula, and $K' = \langle W, R, P'\rangle$.
We prove by induction that $\I_{K'}(B)_\alpha = 1$ iff
$\alpha \models_\AL B$:

$B$ is a variable: $P'(B) = W$ iff $\I_K(B)_\alpha = 1$ for all
$\alpha \in W$ by definition of~$\I_K$.

$B \equiv \neg C$: By the definition of $\widetilde\neg_{\M_K}$,
$\I_K(B)_\alpha = 1$ iff $\I_K(C)_\alpha = 0$.
By induction hypothesis, this is the case iff $\alpha \not\models_\AL C$.
This in turn is equivalent to $\alpha \models_\K B$. Similarly
if $B$ is of the form $C \land D$, $C \lor D$, and $C \impl D$.

$B \equiv \Box C$: $\I_K(B)_\alpha = 1$ iff for all $\beta$
with $\alpha \mathrel{R} \beta$ we have $\I_K(C)_\beta = 1$.
By induction hypothesis this is equivalent to
$\beta \models_\AL C$. But by the definition of $\Box$ this
obtains iff $\alpha \models_\AL B$. Similarly for $\Diamond$.

(1)~Every valuation $\I$ of $\M_K$ defines a function $P_\I$ via
$P_\I(X) = \{\alpha \mid \I(X)_\alpha = 1\}$. Obviously,
$\I = \I_{P_\I}$.
If $\AL \models B$, then $\langle W, R, P_\I\rangle \models_\AL B$.
By the preceding argument then $\I(B)_\alpha = 1$ for
all $\alpha \in W$. Hence, $B$ takes the
designated value under every valuation.

(2)~Suppose $A$ is not true in $K$. This is the case only if there is
a world $\alpha$ at which it is not true. Consequently,
$\I_K(A)_\alpha = 0$ and $A$~takes a non-designated truth value
under~$\I_K$.
\qed\end{proof}

The above method can be used to construct
many-valued logics from Kripke structures for not only modal
logics, but also for intuitionistic logic.
Kripke semantics for \IPL{} are defined analogously,
with the exception that $\alpha \models A \impl B$
iff $\beta \models A \impl B$ for all $\beta \in W$ s.t.{}
$\alpha \mathrel{R} \beta$.  \IPL{} is then characterized by the
class of all finite trees \cite[Ch.~4, Thm.~4(a)]{Gabbay:81}.
Note, however, that for intuitionistic Kripke semantics
the form of the {\em assignments}~$P$ is restricted:
If $w_1 \in P(X)$ and $w_1 \mathrel{R} w_2$ then
also $w_2 \in P(X)$ \cite[Ch.~4, Def.~8]{Gabbay:81}.
Hence, the set of truth values has to be restricted in a similar way.
Usually, satisfaction for intuitionistic Kripke semantics
is defined by satisfaction in the {\em initial} world.
This means that every sequence where the first entry equals~1
should be designated.  By the above restriction,
the only such sequence is the constant 1-sequence.

\begin{example}
The Kripke tree with three worlds
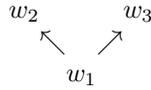
\[
\begin{array}{r@{}c@{}c@{}c@{}l}
w_2 & & & & w_3 \\
 & \nwarrow & & \nearrow \\
& & w_1
\end{array}
\]  
yields a five-valued logic~${\bf T}_3$, with
$V({\bf T_3}) = \{000, 001, 010, 011, 111\}$,
$V^+({\bf T_3}) = \{111\}$,
the truth table
for implication 
\[
\begin{array}{c|cccccccc}
\impl & 000 & 001 & 010 & 011 &  111 \\
\hline
000 & 111 & 111 & 111 & 111 &  111 \\
001 & 010 & 111 & 010 & 111 &  111 \\
010 & 001 & 001 & 111 & 111 &  111 \\
011 & 000 & 001 & 010 & 111 &  111 \\
111 & 000 & 001 & 010 & 011 &  111
\end{array}\]
$\bot$ is the constant $000$, $\neg A$ is defined
by $A \impl \bot$, and $\lor$ and $\land$ are
given by the componentwise classical operations.

The Kripke chain with four worlds corresponds
directly to the five-valued G\"odel logic~${\bf G}_5$.
It is well know that $(X \impl Y) \lor (Y \impl X)$
is a tautology in all \Gm.  Since ${\bf T}_3$
falsifies this formula (take $001$ for $X$ and
$010$ for $Y$), we know that ${\bf G}_5$
is not the best five-valued approximation of~\IPL.

Furthermore, let 
\begin{eqnarray*}
O_5 & = & \bigwedge_{1 \le i < j \le 5} (X_i \impl X_j) \lor (X_j \impl X_i) {\rm \ and}\\
F_5 & = & \bigvee_{1 \le i < j \le 5} (X_i \impl X_j).
\end{eqnarray*}
$O_5$ assures that the truth values assumed by $X_1$, \ldots, $X_5$ are
linearly ordered by implication.  Since neither $010 \impl 001$ nor $001 \impl 010$
is true, we see that there are only four truth values which can
be assigned to $X_1$, \dots,~$X_5$ making $O_5$ true.
Consequently, $O_5 \impl F_5$ is valid in ${\bf T}_3$.
On the other hand, $F_5$ is false in ${\bf G}_5$.
\end{example}
 
\begin{theorem}\label{thm:fmpapprox}
Let \AL{} be a modal logic characterized by a set of finite Kripke
models~$\K = \{K_1, K_2, \ldots\}$. A sequential approximation of~\AL{}
is given by  $\langle \M_1, \M_2, \ldots \rangle$ where $\M_1 =
\M_{K_1}$, and $\M_{i+1} = \M_i \times \M_{K_{i+1}}$.  This
approximation is effective if \K{} is effectively enumerable.
\end{theorem}

\begin{proof}
(1)~$\Taut(\M_i) \supseteq \Taut(\AL)$:  By induction on $i$:  For $i =
1$ this is Lemma~\ref{lem:MKmodels}~(1).  For $i > 1$ the statement
follows from Lemma~\ref{lem:tauttimes}, since $\Taut(\M_{i-1})
\supseteq \Taut(\AL)$ by induction hypothesis, and $\Taut(\M_{K_i})
\supseteq \Taut(\AL)$ again by Lemma~\ref{lem:MKmodels}~(1).

(2)~$\M_i \bettereq \M_{i+1}$ from $A \cap B \subseteq A$ and
Lemma~\ref{lem:tauttimes}.

(3)~$\Taut(\AL) = \bigcap_{i\ge 1} \Taut(\M_i)$.  The
$\subseteq$-direction follows immediately from~(1).  Furthermore, by
Lemma~\ref{lem:MKmodels}~(2), no non-tautology of~\AL{} can be a member
of all $\Taut(\M_i)$, whence $\supseteq$ holds.
\qed\end{proof}

\begin{remark}  Finitely axiomatizable modal logics with the fmp
always have an effective sequential approximation, since it is then
decidable if a given finite Kripke structure satisfies the
axioms. Urquhart \cite{Urquhart:81} has shown that this is not true if
the assumption is weakened to recursive axiomatizability, by giving an
example of an undecidable recursively axiomatizable modal logic with
the fmp.  Since this logic cannot have an effective sequential
approximation, its characterizing family of finite Kripke models is
not effectively enumerable.  The preceding theorem thus also shows
that the many-valued closure of a calculus for a modal logic with the
fmp equals the logic itself, provided that the calculus contains modus
ponens and necessitation as the only rules. (All standard
axiomatizations are of this form.)
\end{remark}

\section{Conclusion}

Our brief discussion unfortunate must leave many interesting questions
open, and suggests further questions which might be topics for future
research.  The main open problem is of course whether the approach
used here can be extended to the case of first-order logic.  There are
two distinct questions: The first is how to check if a given
finite-valued matrix is a cover for a first-order calculus.  Is this
decidable?  One might expect that it is at least for ``standard''
formulations of first-order rules, e.g., where the rules involving
quantifiers are monadic in the sense that they only involve one
variable per rule.  The second question is whether the relationship
$\better^*$ is decidable for $n$-valued first order logics.
Another problem, especially in view of possible applications in
computer science, is the complexity of the computation of optimal
covers.  One would expect that it is tractable at least for some
reasonable classes of calculi which are syntactically characterizable.

We have shown that for strictly analytic calculi, the many-valued
closure coincides with the set of theorems, i.e., that they are
effectively approximable by their finite-valued covers.  Is it
possible to extend this result to a wider class of calculi, in
particular, what can be said about calculi in which modus ponens is
the only rule of inference (so-called Frege systems)?  For calculi
which are not effectively approximable, it would still be interesting
to characterize the many-valued closure.  For instance, we have seen
that the many-valued closure of linear logic is not equal to linear
logic (since linear logic is undecidable) but also not trivial (since
all classical non-tautologies are falsified in a 2-valued cover).
What is the many-valued closure of linear logic?  For those (classes
of) logics for which we have shown that sequential approximations are
possible, our methods of proof also do not yield optimal solutions.
For instance, for modal logics with the finite model property we have
shown that all non-valid formulas can be falsified in the many-value
logic obtained by coding the corresponding Kripke countermodel.  But
there may be logics with fewer truth-values which also falsify these
formulas.  A related question is to what extent our results on
approximability still hold if we restrict attention to many-valued
logics in which only one truth-value is designated.  The standard
examples of sequences of finite-valued logics approximating, e.g., \L
ukasiewicz or intuitionistic logic are of this form, but it need not
be the case that every approximable logic can be approximated by
logics with only one designated value.

\end{document}